\documentclass[12pt,oneside,openany,article]{memoir}

\usepackage{mempatch}

\nouppercaseheads
\usepackage[amsthm,thmmarks]{ntheorem}
\usepackage[noTeX]{mmap}
\usepackage{etex}

\usepackage[english]{babel}
\usepackage[leqno]{mathtools}

\usepackage{mathrsfs}
\usepackage{upgreek}
\usepackage[compress,square,comma,numbers]{natbib}
\usepackage{hypernat}

\usepackage{eso-pic}
\usepackage{sfmath}
\usepackage{tikz}
\usepackage{todonotes}

\usepackage{subfig}

\usepackage{microtype}

\usepackage{array}

\usepackage{graphicx,xcolor}
\definecolor{refkey}{rgb}{0,0,1}
\definecolor{labelkey}{rgb}{1,0,0}

\usepackage{hyperxmp}
\usepackage{xr-hyper}
\usepackage{nameref}
\usepackage[pdftex,bookmarks,pdfnewwindow,plainpages=false,unicode]{hyperref}

\usepackage{bookmark, url}

\usepackage{enumitem}
\usepackage{amssymb}

\newenvironment{claim}[1][{\textup{(\theequation)}}]{\refstepcounter{equation}\vglue10pt
\begin{trivlist}
\item[{\hskip\labelsep#1}]}{\vglue10pt\end{trivlist}}

\hypersetup{
colorlinks=true,
linkcolor=black,
citecolor=black,
urlcolor=blue,
pdfauthor={Victor Ivrii},
pdftitle={Spectral asymptotics for Dirichlet to Neumann operator},
pdfsubject={Sharp Spectral Asymptotics},
pdfkeywords={Microlocal Analysis,  Sharp Spectral Asymptotics, Periodic flows},
bookmarksdepth={4}
}

\theoremstyle{plain}
\newtheorem{theorem}{Theorem}[chapter]

\newtheorem{proposition}[theorem]{Proposition}
\newtheorem{corollary}[theorem]{Corollary}

\theoremstyle{definition}
\newtheorem{definition}[theorem]{Definition}

\theoremstyle{remark}
\newtheorem{remark}[theorem]{Remark}

\newtheorem{conjecture}[theorem]{Conjecture}

\makeatletter
\newtheoremstyle{plainfoot}%
  {\item[\hskip\labelsep \theorem@headerfont ##1\ ##2\,\footnotemark\theorem@separator]}%
  {\item[\hskip\labelsep \theorem@headerfont ##1\ ##2\ (##3)\, \footnotemark\theorem@separator]}
\makeatother

\theoremstyle{plainfoot}
\theoremseparator{.}
\newtheorem{theorem-foot}[theorem]{Theorem}
\newtheorem{lemma-foot}[theorem]{Lemma}
\newtheorem{proposition-foot}[theorem]{Proposition}
\newtheorem{corollary-foot}[theorem]{Corollary}
\newtheorem{conjecture-foot}[theorem]{Conjecture}
\newtheorem{condition-foot}[theorem]{Condition}

\theoremstyle{plainfoot}
\theoremseparator{.}
\theorembodyfont{}
\newtheorem{definition-foot}[theorem]{Definition}
\newtheorem{Problem-foot}[theorem]{Problem}

\theoremstyle{plainfoot}
\theoremseparator{.}
\theorembodyfont{}
\theoremheaderfont{\slshape}
\newtheorem{remark-foot}[theorem]{Remark}         
\newtheorem{example-foot}[theorem]{Example}
\newtheorem{problem-foot}[theorem]{Problem}

\numberwithin{equation}{chapter}

\newenvironment{phantomequation}[1][]{\refstepcounter{equation}}{}

\newcounter{note}

\newcommand{\loc}{\mathsf{loc}}

\newcommand{\inn}{\mathsf{inn}}
\newcommand{\out}{\mathsf{out}}
\newcommand{\sym}{\mathsf{sym}}
\newcommand{\asym}{\mathsf{asym}}

\newcommand{\bR}{\mathbb{R}}

\newcommand{\cI}{\mathcal{I}}
\newcommand{\cJ}{\mathcal{J}}
\newcommand{\cY}{\mathcal{Y}}
\newcommand{\cX}{\mathcal{X}}

\newcommand{\cK}{\mathcal{K}}

\newcommand\sC{\mathscr{C}}
\newcommand\sL{\mathscr{L}}
\newcommand\sH{\mathscr{H}}

\newcommand{\dist}{\operatorname{dist}}
\newcommand{\mes}{\operatorname{mes}}

\newcommand\supp{\operatorname{supp}}

\newcommand{\Ker}{\operatorname{Ker}}
\newcommand{\Ran}{\operatorname{Ran}}
\newcommand{\Spec}{\operatorname{Spec}}

\newcommand{\N}{\mathsf{N}}
\newcommand{\D}{\mathsf{D}}
\newcommand{\w}{\mathsf{w}}
\newcommand{\T}{\mathsf{T}}
\newcommand{\W}{\mathsf{W}}

\renewcommand{\Re}{\operatorname{Re}}

\addtopsmarks{headings}{}{%
\createmark{chapter}{right}{shownumber}{}{. \ }
}
\title{Spectral asymptotics for Dirichlet to Neumann operator in the domains with edges
}

\author{Victor Ivrii\thanks{This research was supported in part by National Science and Engineering  Research Council (Canada) Discovery Grant  RGPIN 13827}}

\begin{document}

\maketitle

\begin{abstract}
We consider eigenvalues of the Dirichlet-to-Neumann operator for Laplacian in the domain (or manifold) with edges and establish the asymptotics of the eigenvalue counting function 
\begin{equation*}
\N(\lambda)= \kappa_0\lambda^d  +O(\lambda^{d-1})\qquad \text{as\ \ } \lambda\to+\infty,
\end{equation*}
where $d$ is dimension of the boundary. Further, in certain cases we establish two-term asymptotics 
\begin{equation*}
\N(\lambda)= \kappa_0\lambda^d  +\kappa_1\lambda^{d-1}+o(\lambda^{d-1})\qquad \text{as\ \ } \lambda\to+\infty.
\end{equation*}
We also establish improved asymptotics for Riesz means.
\end{abstract}

\enlargethispage{2.5\baselineskip}

\chapter{Introduction}
\label{sect-1}

Let $X$ be a compact connected $(d+1)$-dimensional Riemannian manifold with the boundary $Y$, regular enough to properly define operators $J$ and $\Lambda$ below\footnote{\label{foot-1} Manifolds with edges are of this type.}.
Consider \emph{Steklov problem\/}
\begin{align}
&\Delta w=0 &&\text{in}\ \ X,\label{eqn-1-1}\\
&(\partial_\nu +\lambda) w|_Y  =0, \label{eqn-1-2}
\end{align}
where $\Delta$ is the positive Laplace-Beltrami operator\footnote{\label{foot-2} Defined via quadratic forms.}, acting on functions on $X$, and $\nu$ is the unit inner normal to  $Y$. In the other words, we consider eigenvalues of the Dirichlet-to-Neumann operator. For $v$, which is a restriction to $Y$ of $\sC^2$ function, we define  $Jv= w$, where 
$\Delta w=0$ in $X$, $w|_Y=v$,  \underline{and} $\Lambda v=-\partial_\nu Jv|_Y$.  

\begin{definition}\label{def-1-1}
$\Lambda$ is called \emph{Dirichlet-to-Neumann operator\/}.
\end{definition}

The purpose of this paper is to consider manifold with the boundary  which has edges: i.e. each point $y\in Y$ has a neighbourhood $U$ in $\bar{X}\coloneqq X\cap Y$, which is the diffeomorphic either to $\bR^+ \times \bR^d$ (then $y$ is a \emph{regular point\/}), or to $\bR^{+\,2}  \times \bR^{d-1}$ (then $y$ is an \emph{inner edge point\/}) or to $(\bR^2 \setminus \bR^{-\,2})  \times \bR^{d-1}$ (then $y$ is an \emph{outer edge point\/}). Let $Z_\inn$ and $Z_\out$ be sets of the inner and outer edge points respectively, and $Z=Z_\inn\cup Z_\out$.

One can prove easily the following proposition:

\begin{proposition}\label{prop-1-2}
\begin{enumerate}[label=(\roman*), wide, labelindent=0pt]
\item\label{propo-1-2-i}
$\Lambda$ is a non-negative essentially self-adjoint operator in $\sL^2(Y)$; $\Ker(\Lambda)$ consists of constant functions.
\item\label{propo-1-2-ii}
$\Lambda$ has a discrete accumulating to infinity spectrum with eigenvalues $0=\lambda_0<\lambda_1\le \ldots $ could be obtained recurrently from the following variational problem:
\begin{multline}
\int _X |\nabla w|^2\,dx\mapsto \min (=\lambda_n)\\
\text{as\ \ } \int_Y |w|^2\,dx'=1, \qquad \int_Y w w^\dag_k\,dx'=0\quad\text{for\ \ } k=0,\ldots, n-1.
\label{eqn-1-3}
\end{multline}
\end{enumerate}
\end{proposition}
\enlargethispage{\baselineskip}

\begin{corollary}\label{cor-1-3}
The number of eigenvalues of $\Lambda$, which are less than $\lambda$, equals to the maximal dimension of the linear space of $\sC^2$-functions, on which the quadratic form
\begin{equation}
\int _X |\nabla w|^2\,dx-\lambda \int_Y |w|^2\,dx'
\label{eqn-1-4}
\end{equation}
is negative definite.
\end{corollary}

\begin{proposition}\label{prop-1-4}
Operator $\Lambda$ has a domain $\sH^1(Y)$ and 
\begin{equation}
\|\Lambda u\|_{Y}+\|u\|_{Y}\asymp \|u\|_{\sH^1(Y)},
\label{eqn-1-5}
\end{equation}
 where $(.,.)$ and $\|.\|$ denote  $\sL^2$ inner product and norm.
\end{proposition}

\begin{proof}
Let  $L=\ell\cdot \nabla$, $\ell$ be  a vector field which makes an acute  angle with the inner normal (at $Z$--with both inner normals). Consider 
\begin{multline}
0=-(\Delta w, L w)_X= (\nabla w, \nabla L w)_X + (\partial_\nu  w, L w)_Y=\\
\int  Q(\nabla w)\,dy + O(\|w\|^2_{\sH^1(X)}),
\label{eqn-1-6}
\end{multline}
where 
\begin{equation}
Q(\nabla w)= (\nu\cdot \nabla w)(\ell\cdot\nabla w )-\frac{1}{2} \nu\cdot\ell|\nabla w|^2.
\label{eqn-1-7}
\end{equation}
This quadratic form has one  positive  and $d$ negative eigenvalues. Further, on the subspace orthogonal to $\ell$, all eigenvalues are negative. Then 
\begin{equation}
\|\partial _\nu w\|^2 + C\|w\|_{\sH^1(X)}^2 \asymp \|w\|_{\sH^(Y)}^2.
\label{eqn-1-8}
\end{equation}
Combined with the estimate for $\|w\|_{\sH^1(X)}^2\le C\|w\|_{\sH^{\frac{1}{2}}(Y)}^2 $ it implies the statement.
\end{proof}

\begin{remark}\label{rem-1-5}
\begin{enumerate}[label=(\roman*), wide, labelindent=0pt]
\item\label{rem-1-5-i}
If $Y$ is infinitely smooth, then $\Lambda$ is the first-order pseudodifferential operator on $Y$ with the principal symbol $(g_Y(x,\xi))^{1/2}$, where $g_Y$ is the restriction of the metrics to $Y$. Then the standard results hold:
\begin{equation}
\N(\lambda)=\kappa_0 \lambda^d + O(\lambda^{d-1})\qquad\text{as\ \ }\lambda \to +\infty
\label{eqn-1-9}
\end{equation}
with the standard coefficient $\kappa_0 =(2\pi)^{-d}\omega_d \mes  (Y)$, where $\mes(Y)$ means $d$-dimensional volume of $Y$, $\omega_d$ is the volume of the unit ball in $\bR^d$. We also can get two-term asymptotics with the same remainder estimate for $\N(\lambda)*\lambda_+^{r-1}$, $0<r \le 1$.

\item\label{rem-1-5-ii}
Moreover, if the set  of all periodic geodesics of $Y$ has measure $0$, then
\begin{equation}
\N(\lambda)=\kappa_0 \lambda^d + \kappa_1 \lambda^{d-1}+o(\lambda^{d-1})\qquad\text{as\ \ }\lambda \to \infty.
\label{eqn-1-10}
\end{equation}
We also can get two-term asymptotics (three-term for $r=1$) with the same remainder estimate for 
$\N(\lambda)*\lambda_+^{r-1}$, $0<r \le 1$.
The same asymptotics, albeit with a larger number of terms, hold for $r>1$.

\item\label{rem-1-5-iii}
``Regular'' singularities of the dimension $<(d-1)$ (like conical points in $3\D$) do not cause any problems for asymptotics of $\N(\lambda)$---we can use a rescaling technique to cover them; moreover, in the framework of this paper they would not matter even combined with edges (like vertices in $3\D$). 
\end{enumerate}
\end{remark}

\chapter{Dirichlet-to-Neumann operator}
\label{sect-2}

\section{Toy-model:  dihedral angle}
\label{sect-2-1}

Let $Z=\bR^{d-1}$ with the Euclidean metrics, $X=  \cX \times Z $, $Y=  \cY\times Z$, where $\cX$ is a planar angle of solution $\alpha$, $0<\alpha\le 2\pi$, $\cY=\cY_1\cup \cY_2$, $\cY_j$ are rays (see Figure~\ref{fig-2}).

Then one can identify $Y$ with $\bR^d$ with coordinates $(s,z)$, where $z\in Z$ and 
\begin{itemize}[label=-]
\item $s=\dist (y,Z)$ for for a point   $y\in Y_1= \cY_1\times Z$, 
\item
$s=-\dist (y,Z)$ for for a point  $y\in Y_2= \cY_2\times Z$. 
\end{itemize}

Then we have a Euclidean metrics and a corresponding positive Laplacian $\Delta_Y$ on $Y$.

\begin{remark}\label{rem-2-1}
\begin{enumerate}[label=(\roman*), wide, labelindent=0pt]
\item\label{rem-2-1-i}
We can consider any angle $\alpha >0$, including $\alpha >2\pi$ (in which case $X$ could be defined in the polar coordinates, but then we need to address some issues with the domain of operator).

\item\label{rem-2-1-ii}
If $\alpha=\pi$, then $\Lambda=\Delta_Y^{1/2}$. 

\item\label{rem-2-1-iii}
We say that $X$ is a \emph{proper angle\/} if $\alpha\in (0,\pi)$ and that $X$ is a \emph{improper angle\/} if $\alpha\in (\pi,2\pi)$. We are not very concerned about $\alpha=\pi,\,2\pi$ since these cases will be forbidden in the general case.
\end{enumerate}
\end{remark}

For this toy-model we can make a partial Fourier transform $F_{z\to \zeta}$ and then study equation in the planar angle:
\begin{equation}
\Delta_2 w+w=0,
\label{eqn-2-1}
\end{equation}
where $\Delta_2$ is a positive $2\D$-Laplacian and we made also a change of variables $x''\mapsto |\zeta|\cdot x''$, $x''=(x_1,x_2)$.
Denote by $\bar{J}$ and $\bar{\Lambda}$ operators $J$ and $\Lambda$ for (\ref{eqn-2-1}). This problem is extensively studied in Appendix~\ref{sect-A}.

Then we can use the separation of variables.  Singularities at the vertex for solutions to (\ref{eqn-2-1}) and 
$w|_Y=0$, are the same as for $\Delta _2w=0$, $w|_Y=0$ and they are combinations of 
$r^{\pi n/\alpha}\sin (\pi n\theta/\alpha)$ with $n=1,2,\ldots$, where 
$(r,\theta)\in \bR^+ \times (0,\alpha)$ are polar coordinates.

This show the role of $\alpha$: if $\alpha \in (0,\pi)$ those functions are  in $\sH^{\sigma}_\loc(\cX)$ with 
$\sigma<  1+\pi n/\alpha$, and  $\partial_\nu w|_Y$ belong to $\sH^{\sigma-3/2}_\loc(\cY)$.

One can prove easily the following Propositions~\ref{prop-2-2} and \ref{prop-2-3} below:

\begin{proposition}\label{prop-2-2}
The following are bounded operators
\begin{align}
& \Delta_\D^{-1}:\sH^{\sigma} (X)\to \sH^{\sigma+2}(X),
\label{eqn-2-2}\\
&J:\sH^{\sigma+\frac{3}{2}}(Y)\to \sH^{\sigma+2}(X),
\label{eqn-2-3}\\
&\Lambda :\sH^{\sigma+\frac{3}{2}}(Y)\to \sH^{\sigma+\frac{1}{2}}(X),
\label{eqn-2-4}
\end{align}
where $\Delta_\D$ is an  operator $\Delta$ with zero Dirichlet boundary conditions on $Y$ and
\begin{itemize}[label=-]
\item $\sigma\in [-\frac{1}{2},0]$, if $\alpha \in (0,\pi)$, and
\item $\sigma\in [-\frac{1}{2},\bar{\sigma})$ with  $\bar{\sigma}=\pi /\alpha-1$ otherwise.
\end{itemize}
\end{proposition}

\begin{proposition}\label{prop-2-3}
For equation \textup{(\ref{eqn-2-1})} in $\cX$ 
\begin{equation}
\bar{\Lambda}- (D_s^2+1)^{1/2}  = \sum_{j+k\le 1} D_s^j \bar{K}_{jk}D_s^k\,,
\label{eqn-2-5}
\end{equation}
where operators $\bar{K}_{jk}$ have Schwartz kernels $\bar{K}_{jk}(s,s')$ such that
\begin{multline}
|D_s^pD_{s'}^q \bar{K}_{jk}(s,t') |\le \\
C_{pqm }|s|^{-(\bar{\sigma}-p)_-}|s'|^{-(\bar{\sigma}-q)_-}(|s|+|s'|)^{-p-q +(\bar{\sigma}-p)_- + (\bar{\sigma}-q)_-}
(|s|+|s'|+1)^m
 \label{eqn-2-6}
\end{multline}
and $l_\pm \coloneqq \max (\pm l,0)$ and $m$ is arbitrarily large.
\end{proposition}

Then

\begin{corollary}\label{cor-2-4}
For the toy-model in $X$
\begin{equation}
\Lambda- \Delta_Y^{1/2}  = \sum_{j+k\le 1} D_s^j K_{jk}D_s^k\,,
 \label{eqn-2-7}
\end{equation}
where operators $K_{jk}$ have Schwartz kernels 
\begin{equation}
K_{j,k}(x',s; y', s')=
(2\pi)^{1-d} \iint |\xi'|^{2-j-k} \bar{K}_{jk}(s|\xi'|,\, s'|\xi'|) e^{-i\langle x'-y',\xi'\rangle}\,d\xi'.
\label{eqn-2-8}
\end{equation}
\end{corollary}

\section{General case}
\label{sect-2-2}

Consider now the general case. In this case we can again introduce coordinate $s$ on $Y$ and consider $Y$ as a Riemannian manifold, but with the metrics which is only $\sC^{0,1}$ (Lipschitz class); more precisely, it is 
$\sC^\infty$ on both $Y_1$ and $Y_2$, but the first derivative with respect to $s$ may have a jump on $Z$. It does not, however, prevent us from introduction of $\Delta_Y$ and therefore $\Delta_Y^{\frac{1}{2}}$, but the latter would not be necessarily the classical pseudodifferential operator.

We want to exclude the degenerate cases of the angles $\pi$ and $2\pi$. So, let us assume that
\begin{claim}\label{eqn-2-9}
$Z=\{x\colon x_1=x_2=0\}$ and $X=Z\times \cX$ with a planar angle  $\cX\ni (x_1,x_2)$, disjoint from half-plane and the plane with a cut. 
\end{claim}

\begin{definition}\label{def-2-5}
For $z\in Z$ let $\alpha(z)$ be an internal angle between two leaves of $Y$ at point $z$ (calculated in the corresponding metrics). Due to our assumption either $\alpha(z) \in (0,\pi)$ or $\alpha (z)\in (\pi,2\pi)$. Let $Z_j$ be a connected component of $Z$.
\begin{enumerate}[label=(\roman*), wide, labelindent=0pt]
\item\label{def-2-5-i}
 $Z_j$ is a \emph{inner edge} if $\alpha(z)\in (0,\pi)$ on $Z_j$, and  
 \item\label{def-2-5-ii}
 $Z_j$ is an \emph{outer edge} if $\alpha(z)\in (\pi,2\pi)$ on $Z_j$. 
 \end{enumerate}
 \end{definition}

One can prove easily

\begin{proposition}\label{prop-2-6}
The following are bounded operators
\begin{align}
& \Delta_\D^{-1}:\sH^{\sigma} (X)\to \sH^{\sigma+2}(X),
\label{eqn-2-10}\\
&J:\sH^{\sigma+\frac{3}{2}}(Y)\to \sH^{\sigma+2}(X),
\label{eqn-2-11}\\
&\Lambda :\sH^{\sigma+\frac{3}{2}}(Y)\to \sH^{\sigma+\frac{1}{2}}(X),
\label{eqn-2-12}
\end{align}
where $\Delta_\D$ is an  operator $\Delta$ with zero Dirichlet boundary conditions on $Y$ and
\begin{enumerate}[label=(\roman*), wide, labelindent=0pt]
\item\label{prop-2-6-i}
$\sigma\in [-\frac{1}{2},0]$, if $\alpha(z)\in (0,\pi)\ \forall z\in Z$, and
\item\label{prop-2-6-ii}
$\sigma\in [-\frac{1}{2},\bar{\sigma})$ with  $\bar{\sigma}=\pi /\bar{\alpha}-1$, $\bar{\alpha}=\max_{z\in Z} \alpha(z)$ otherwise.
\end{enumerate}
\end{proposition}

One can also prove easily

\begin{proposition}\label{prop-2-7}
In the general case, assuming that $Z=\{x\colon x_1=x_2=0\}$ and $X=Z\times \cX$ with a planar angle 
$\cX\ni (x_1,x_2)$ of solution $\in (0,\pi)\cup(\pi,2\pi)$
\begin{equation}
\Lambda- \Delta_Y^{1/2}  = b+ \sum_{j+k\le 1} D_s^j K_{jk}D_s^k\,,
 \label{eqn-2-13}
\end{equation}
where $b$ is a bounded operator and operators $K_{jk}$ have Schwartz kernels and
\begin{multline}
K_{j,k}(x',s; y', s')=\\
(2\pi)^{1-d} \iint |\xi'|^{2-j-k} \bar{K}_{jk}\bigl(\frac{1}{2}(x'+y'), s|\xi'|,\, s'|\xi'|\bigr)\, e^{-i\langle x'-y',\xi'\rangle}\,d\xi'.
\label{eqn-2-14}
\end{multline}
\end{proposition}

\begin{remark}\label{rem-2-8}
On the distances $\gtrsim 1$ from $Z$, $b$ is a classical $0$-order pseudodifferential operator, on the distance $\gtrsim |\xi'|^{-1+\delta}$ it is a rough $0$-order pseudodifferential operator\footnote{\label{foot-3} I. e. with the symbols such that 
$|D_x^\alpha D_\xi^\beta|\le C_{\alpha\beta}  \rho^{-|\beta|}\gamma^{-|\alpha}$ with $\rho\gamma \ge h^{1-\delta}$, $\rho\gamma\ge h^{1-\delta}$. Here $\rho=1$, $\gamma= |x|$.}.
\end{remark}

\chapter{Microlocal analysis}
\label{sect-3}

\section{Propagation of singularities near edge}
\label{sect-3-1}

We are going to consider microlocal analysis near point 
$(\bar{x},\bar{\xi}'')\in T^*Z$ under assumption (\ref{eqn-2-9}). In our approach we use definition of operator $\Lambda$ rather than its description of the previous Section~\ref{sect-2}.  So, let $x=(x'';x')\in \bR^2\times \bR^{d-1}$.

\begin{proposition}\label{prop-3-1}
\begin{enumerate}[label=(\roman*), wide, labelindent=0pt]
\item\label{prop-3-1-i}
Let $q_j=q_j(\xi')$ ($j=1,2$) be two symbols, constant as $|\xi'|\ge C$. Assume that the 
$\dist (\supp (q_1),\supp(q_2))\ge \epsilon$. 

Consider $h$-pseudodifferential operators  $Q_j=q_j^\w(h^{-1}D')$, $j=1,2$. Then the operator norms of
\begin{phantomequation}\label{eqn-3-1}\end{phantomequation}
\begin{align}
&Q_1 \Delta_\D^{-1}Q_2:\sL^2(X)\to \sH^2(X),&&Q_1 JQ_2:\sH^{\frac{1}{2}}(Y)\to \sH^2(X),
\tag*{$\textup{(\ref*{eqn-3-1})}_{1,2}$}\label{eqn-3-1-1}\\
&Q_1\Lambda Q_2 :\sH^{1}(Y)\to \sL^2(Y)
\tag*{$\textup{(\ref*{eqn-3-1})}_{3}$}\label{eqn-3-1-3}
\end{align}
do not exceed $C'h^s$ with arbitrarily large $s$ where $\Delta_\D$ is an  operator $\Delta$ with zero Dirichlet boundary conditions on $Y$.
\item\label{prop-3-1-ii}
Let $Q_j(x')$ ($j=1,2$) be two functions. Then operators $\textup{(\ref{eqn-3-1})}_{1-3}$ are infinitely smoothing by $x'$. 

\end{enumerate}
\end{proposition}

\begin{proof}
\begin{enumerate}[label=(\roman*), wide, labelindent=0pt]
\item\label{pf-3-1-i}
Without any loss of the generality one can assume that $q_j$ are constant also in the vicinity of $0$. Then 
the operator norms of $[Q_j,\Delta]\Delta_\D^{-1}$ in $\sL^2(X)$ do not exceed $Ch$; replacing $Q_j$ by $Q_j^{(n)}$ with $Q_j^{(0)}=Q_j$ and $Q_j^{(n)}\coloneqq [Q_j^{(n-1)},\Delta]\Delta_\D^{-1}$ for $j=1,2,\ldots$, we prove by induction that the operator norms of $Q_j^{(n)}$ in in $\sL^2(X)$ do not exceed $Ch^n$. Then one can prove by induction easily that the operator norm of $\textup{(\ref*{eqn-3-1})}_{1}$ does not exceed $Ch^s$. 

Then one can prove easily that the operator norm of $\textup{(\ref{eqn-3-1})}_{2,3}$ do not exceed $Ch^s$ as  well. It concludes the proof of Statement~\ref{prop-3-1-i}.

\item\label{pf-3-1-ii}
Statement~\ref{prop-3-1-ii} is proven by the same way. 
\end{enumerate}
\vskip-\baselineskip\ 
\end{proof}

Let  $u(x,y,t)$ be Schwartz kernel of $e^{it\Lambda}$, $x,y\in Y$. 

\begin{proposition}\label{prop-3-2}
Consider $h$-pseudodifferential operator  $Q=q^\w(x',h^{-1}D')$ where $q$ vanishes $\{|\xi'|\le c_0\}$. Let  $\chi\in \sC_0^\infty (\bR)$, $T\ge h^{1-\delta}$. 
Then operator norms of 
$F_{t\to \tau} \chi_T(t) Q_x u$ and $F_{t\to \tau} \chi_T(t) u\,^t\!Q_y$
do not exceed $C'_Th^s$ for $\tau\le c$ for $c_0=c_0(c)$.
\end{proposition}
 
\begin{proof}
One need to consider $v=e^{it\Lambda}f$, $f\in \sH^1 (Y)$, $\|f\|_{\sL^2(Y)}=1$ and observe that it satisfies 
$(D_t-\Lambda )v=0$.
Using (\ref{eqn-1-5}) we see that  operator $(D_t-V)$ is elliptic in $\{|\xi'|\ge c_0, \tau\le c\}$ while Proposition~\ref{prop-3-1} ensures its locality.
\end{proof}

Therefore, in what follows

\begin{remark}\label{rem-3-3}
Studying energy levels $\tau\le c$ we can always apply cut-out domain $\{|\xi'|\ge c_0\}$.
\end{remark}

Now we can study the propagation of singularities. Let us prove that the propagation speed with respects to $x$ and $\xi'$ do not exceed $C_0$. For this and other our analysis we need the following Proposition~\ref{prop-3-4}: 

\begin{proposition}\label{prop-3-4}
For $h$-pseudodifferential operator $Q=q^\w(x,hD')$ the following formula onnecting commutators $[\Delta,Q]$ and $[\Lambda +\partial_\nu, Q]$ holds:
\begin{equation}
-\Re i([\Delta,Q]Jv,Jv)_X= \Re i(([\Lambda,Q]+[\partial_\nu,  Q])v,v)_Y
\label{eqn-3-2}
\end{equation}
\end{proposition}

\begin{proof}
First, consider real valued symbol $q=q(x,\xi')$ and $Q=q^\w(x,hD')$ its Weyl quantization. Let $v$ denote just any function on $Y$ and $V$ its continuation as a harmonic function. Then for $w=Jv$
\begin{multline*}
0=(Q\Delta w,w)_X= (\Delta Qw,w) -([\Delta ,Q]w,w)_X=\\
-([\Delta ,Q]w,w)_X+ (Qw,\Delta w)_X - (\partial_\nu Qw, w)_Y + (Qw, \partial _\nu w)_Y= \\
-([\Delta ,Q]w,w)_X - ( Q\partial_\nu w, w)_Y -( [\partial_\nu,Q] w, w)_Y + (Qw, \partial _\nu w)_Y=\\
-([\Delta ,Q]w,w)_X + ( Q \Lambda v, v)_Y -( [\partial_\nu,Q] v, v)_Y - (v, Q\Lambda v)_Y=\\
-([\Delta ,Q]w,w)_X - (\Lambda Q v,v)_Y-( [\partial_\nu,Q] v, v)_Y,
\end{multline*}
which implies (\ref{eqn-3-2}).
\end{proof}

 Now we can prove that at energy levels $\tau\le c$ the propagation speed with respects to $x$ and $\xi'$ do not exceed $C_0=C_0(c)$.

\begin{proposition-foot}\footnotetext{\label{foot-4} Cf. Theorem~\ref{monsterbook-thm-8-5-6}\ref{monsterbook-thm-8-5-6-i} of \cite{monsterbook}.}\label{prop-3-5}
Let $Q_j=q_j^\w (x,hD')$ and $\dist (\supp (q_1),\supp (q_2))\ge C_0 T$ with fixed $T>0$. Let $\chi\in \sC_0^\infty ([-1,1])$. Then for $\tau \le c$
\begin{equation}
|F_{t\to h^{-1}\tau} \bigl( \chi_T(t) Q_{1x} u \,^t\!Q_{2y} \bigr)|\le Ch^m,
\label{eqn-3-3}
\end{equation}
where here and below  $m$ is an arbitrarily large exponent and $C=C_m$.
\end{proposition-foot}

\begin{proof}
\begin{enumerate}[label=(\roman*), wide, labelindent=0pt]
\item\label{pf-3-5-i}
The proof is the standard one for propagation with respect to $(x',\xi')$: we consider $\phi (x',\xi',t)$ and prove that under the microhyperbolicity condition 
\begin{gather}
\phi_t -\{ |\xi'|,\phi \}\ge \epsilon_0,
\label{eqn-3-4}\\
\shortintertext{which is equivalent to} 
2\phi_t - |\xi'|^{-1}\{|\xi'|^2,\phi \}\ge 2\epsilon_0|\xi'|,
\label{eqn-3-5}
\end{gather}
our standard propagation theorem (see Theorem~\ref{monsterbook-thm-2-1-2} of \cite{monsterbook})
holds, just repeating arguments of its proof,  using equality (\ref{eqn-3-2}) and the fact that 
$\|Jv\|_{\sH^{1/2}(X)} \asymp \|v\|_{\sL^2(Y)}$.

Then we plug $\phi (x',\xi',t)=\psi(x',\xi') -t$ with $|\nabla_{x',\xi'} \psi |\le \epsilon_0$, and prove that (\ref{eqn-3-3}) for $q_j=q_j(x',\xi')$.

\item\label{pf-3-5-ii}
We need also prove that the propagation speed with respect to $(x_1,x_2)$\,\footnote{\label{foot-5} Or, equivalently,  with respect to $s$.} does not exceed $C_0$, but it is easy since  for $|s|\ge \epsilon$,  $\Lambda$ is a first-order pseudodifferential operator with the symbol $|\xi|$.
\end{enumerate}
\vskip-\baselineskip\ 
\end{proof}

\begin{remark}\label{rem-3-6}
In fact, it follows from the proof, that the propagation speed with respect to $x'$ do not exceed $C_0$, and the propagation speed with respect to $\xi'$ does not exceed $C_0|\xi'|$ with $C_0$, which does not depend on restriction $\tau\le c$. Meanwhile, the propagation speed with respect $(x_1,x_2)$ does not exceed $1$. 
\end{remark}

Next we prove that  at energy levels $\tau=1$ the propagation speed with respects to $x'$ in the vicinity of $(0,\bar{\xi}')$ with $|\bar{\xi}'|\ge\epsilon_0$ is at least $\epsilon_1=\epsilon_1(\epsilon_0)$.

\begin{proposition-foot}\footnotetext{\label{foot-6} Cf. Theorem~\ref{monsterbook-thm-8-5-6}\ref{monsterbook-thm-8-5-6-ii} of \cite{monsterbook}.}\label{prop-3-7}
Let $Q_j=q_j^\w (x,hD')$ and 
\begin{equation*}
\dist_{x'} (\supp (q_1),\supp (q_2))\le \epsilon_1T
\end{equation*} 
with fixed $T>0$.  Let $\chi\in \sC_0^\infty ([-1,-\frac{1}{2}]\cup [-\frac{1}{2},-1])$. Then for $|\tau-1|\le \epsilon_0$ \textup{(\ref{eqn-3-3})} holds.
\end{proposition-foot}

\begin{proof}
After propagation theorem mentioned in the proof of Proposition~\ref{prop-3-5} is proven we just plug $\phi (x',\xi'.t)=\psi (x',\xi')-\epsilon t$ with $\xi'\cdot \nabla_{x'} \psi \ge 1$.
\end{proof}

\begin{corollary-foot}\footnotetext{\label{foot-7} Cf. Corollary~\ref{monsterbook-cor-8-5-7}\ref{monsterbook-cor-8-5-7-ii} of \cite{monsterbook}.}\label{cor-3-8}
  In the framework of Proposition~\ref{prop-3-5} consider  $|\tau -1|\le \epsilon$. Then 
 \begin{align}
 &|F_{t\to h^{-1}\tau} \Gamma_{x} \chi_T(t)  u \,^t\!Q_y|\le Ch^{1-d+m}T^{-m}
 \label{eqn-3-6}\\
 \shortintertext{and}
& |F_{t\to h^{-1}\tau} \Gamma_{x} \bigl(\bar{\chi}_{T'}(t)- \bar{\chi}_{T}(t)\bigr) u \,^t\!Q_y|
\le Ch^{1-d+m}T^{-m},
\label{eqn-3-7}
  \end{align}
provided 
$\chi\in \sC_0^\infty ([-1,-\frac{1}{2}]\cup[\frac{1}{2},1])$,  $\bar{\chi}\in \sC_0^\infty ([-1,1])$, $\bar{\chi}=1$ on $[-\frac{1}{2},\frac{1}{2}]$, $h\le T\le T'\le T_0$ with small constant $T_0$.
\end{corollary-foot}

\begin{proof}
For small constant $T$ (\ref{eqn-3-6}) follows directly from Proposition~\ref{prop-3-7}, after one proves easily that we can insert $D_{x''}$, $D_{y''}$ to the corresponding estimate, which is easy.

For $h\le T\le T_0$ we use just rescaling like in the proof of Theorem~\ref{monsterbook-thm-2-1-19} of \cite{monsterbook}.
Finally, (\ref{eqn-3-7}) is obtained by the summation with respect to partition of unity with respect to $t$.
\end{proof}   

This implies immediately

\begin{corollary}\label{cor-3-9}
$\N_h (\tau)$ and $\N_h (\tau) *\tau_+^{\sigma-1}$ are approximated by the corresponding Tauberian expressions with $T\asymp h^{1-\delta}$  with errors 
$O(h^{1-d})$ and $O(h^{1-d+\sigma})$ respectively (as $\tau= 1$ and $h\to +0$).
\end{corollary}

\section{Reflection of singularities from the edge}
\label{sect-3-2}

The results of the previous subsection are sufficient to prove \emph{sharp spectral asymptotics\/} (with the remainder estimate $O(\lambda^{d-1})$), which do not require conditions of the global nature, but insufficient to prove sharper spectral asymptotics (with the remainder estimate $o(\lambda^{d-1})$), which  require conditions of the global nature.

For this more ambitious purpose we need to prove that the singularities propagate along geodesic billiards on the boundary $Y$, reflecting and refracting on the edge $Z$ (so billiards will be branching), and the typical singularity (with $|\xi'|<\tau$) does not stick to $Z$.

To do this we will follow  arguments of Subsection~\ref{monsterbook-sect-8-5-4} of \cite{monsterbook}. Assuming (\ref{eqn-2-9}) consider operator $Q=x_1D_1+x_2D_2-i/2$, which acts along $Y$. As an operator in $\sL^2(Y)$ it is self-adjoint, as an operator in $\sL^2(X)$ it is not, but differs from a self-adjoint operator $Q=x_1D_1+x_2D_2-i$ by $i/2$, which does not affects commutators. 

As a result, repeating the proof of Proposition~\ref{prop-3-4} we arrive to
\begin{claim}\label{eqn-3-8}
Under assumption (\ref{eqn-2-9}) equality (\ref{eqn-3-2}) also holds for the operator $Q=x_1D_1+x_2D_2-i$.
\end{claim}

To apply arguments of the proof of Propositions~\ref{monsterbook-prop-8-5-9} and then \ref{monsterbook-prop-8-5-10} of \cite{monsterbook}, we need to check, if operator $i[\Lambda, Q]$ is positive definite, which in virtue of (\ref{eqn-3-2}) is equivalent to the same property for the form in the left:
\begin{equation}
\Re (i[\Delta, Q]w,w) -\Re (i[\partial_\nu ,Q]w,w)_Y \ge \epsilon \|\nabla w\|^2 \text{for\ \ } w\colon \Delta=0.
\label{eqn-3-9}
\end{equation}
For the toy-model $i[\Delta, Q]= 2(D_1^2+D_2^2)$, $i[\partial_\nu ,Q]= -\partial_\nu$, and 
the form on the left coincides with $\|\nabla w\|^2 - 2\|\nabla 'w\|^2$  on $w$ in question and therefore after Fourier transform $F_{x'\to\zeta}$ and change of variables $x_{1,2}$ it boils down to the inequality
\begin{equation}
\|\nabla w\|^2  -\|w\|^2 \ge \epsilon (\|\nabla w\|^2+\|w\|^2)\qquad \text{for\ \ } w\colon \Delta_2 w+w=0
\label{eqn-3-10}
\end{equation}
for two-dimensional $\Delta_2$, norms and scalar products.

This inequality is explored in Appendix~\ref{sect-A}, and in virtue of Proposition~\ref{cor-A-11} (\ref{eqn-3-10}) holds for $\alpha \in (\pi,2\pi)$. Meanwhile due to Proposition~\ref{prop-A-16} (\ref{eqn-3-10}) fails for $\alpha \in (0,\pi)$.

Therefore we arrive to

\begin{proposition-foot}\footnotetext{\label{foot-8} Cf. Proposition~\ref{monsterbook-prop-8-5-9} of \cite{monsterbook}.}\label{prop-3-10}
Consider two-dimensional toy-model (planar angle) with $\alpha \in (\pi, 2\pi)$.

Let $\psi \in \mathscr{C}^\infty _0 ([-1,1])$, $\psi_\gamma (x)=\psi (x/\gamma)$ and 
$\phi \in \mathscr{C}_0^\infty([-1,1])$, $\tau \ge 1+\epsilon_0$.  Then as
$\gamma\ge h^{1-\delta}$, $T\ge C_0\gamma$,
$h^\delta\ge \eta\ge h^{1-\delta}T^{-1}$
\begin{equation}
\| \phi (\eta^{-1} (hD_t-\tau)) \psi_\gamma  e^{it\Lambda}\psi_\gamma|_{t=T} \|
 \le CT^{-1}\gamma + Ch^{\delta'} .
\label{eqn-3-11}
\end{equation}
\end{proposition-foot}

\begin{proof}
Proof follows the proof of  Proposition~\ref{monsterbook-prop-8-5-9} of \cite{monsterbook} with $m=1$, and uses equality (\ref{eqn-3-2}) to reduce calculation of the commutator $[\Lambda,Q]$ to the calculation of the commutator $[\Delta,Q]$. \end{proof}

\begin{proposition-foot}\footnotetext{\label{foot-9} Cf. Proposition~\ref{monsterbook-prop-8-5-10} of \cite{monsterbook}.}\label{prop-3-11}
Consider $(d+1)$-dimensional toy-model (dihedral angle) with $\alpha \in (\pi, 2\pi)$.

Let $\psi \in \mathscr{C}^\infty _0 ([-1,1])$, $\psi_\gamma (x)=\psi (x_1/\gamma)$,
$\phi \in \mathscr{C}_0^\infty([-1,1])$, $\varphi \in \mathscr{C}_0^\infty(\mathbb{R}^{d-1})$ supported in $\{|\xi'|\le 1-\epsilon\}$ with $\epsilon>0$. Finally, let
$\gamma\ge h^{1-\delta}$, $T\ge C h^{-\delta}\gamma$,
$h^\delta\ge \eta\ge h^{1-\delta}T^{-1}$. Then
\begin{equation}
\|\phi (\eta^{-1} (hD_t -1)) \varphi (hD') \psi_\gamma (x_1) e^{it\Lambda}\psi_\gamma(x_1)|_{t=T}\| = O(h^m)
\label{eqn-3-12}
\end{equation}
with arbitrarily large $m$.
\end{proposition-foot}

\begin{proof}
Proof follows the proof of  Proposition~\ref{monsterbook-prop-8-5-10} of \cite{monsterbook} with $m=1$.  
\end{proof}

Now we can consider the general case. Consider a point $\bar{z}=(\bar{x},\bar{\xi'})\in T^*Z$, $|\bar{x}'|<1$. 

We can raise  to points
$\bar{z}^\pm=(\bar{x},\bar{\xi}^{\pm})\in T^*Y_{1,2}|_Z$ with $|\bar{\xi}^{\pm} |=1$ and $\iota \bar{z}^\pm =\bar{z}$,
where $\iota (x,\xi)= (x,\xi')\in T^*Z$ for $(x,\xi)\in T^*Y|_Z$.

Consider geodesic trajectories  $\Psi_t  ( \bar{z}^{\pm} )$, going from $\bar{z}^{\pm})$ into $T^*Y_{1,2}$ for $t<0$, 
$|t|<\epsilon$; this distinguishes these two points. 

We also can consider geodesic trajectories 
$\Psi_t  (\bar{z}^{\pm})$, going from $\bar{z}^{\mp}$ into $T^*Y_{1,2}$ for $t>0$, $|t|<\epsilon$. 

Let  $\iota^{-1}\bar{z}=\{\bar{z}^+,\bar{z}^-\}$ and let $\Psi_t (\iota^{-1}\bar{z})$ be obtained as a corresponding union as well\footnote{\label{foot-10} So we actually restrict $\iota$ to $S^*Y|_Z$ and $\iota^{-1}$ to $B^*Z$.}. So, for such point $\bar{z}$ \ $\Psi_t (\iota^{-1}z)$ with $t<0$ consists of two 
 incoming geodesic trajectories, $\Psi_t (\iota^{-1}z)$ with $t<0$ consists of two outgoing geodesic trajectories. Similarly, for $z\in (T^*(Y\setminus Z))$ we can introduce $\Psi_t (z)$: when trajectory hits $Z$ it branches.

\begin{theorem-foot}\footnotetext{\label{foot-11} Cf. Theorem~\ref{monsterbook-thm-8-5-11} of \cite{monsterbook}.}
\label{thm-3-12} 
Consider a point $z=(x,\xi)\in T^*Y$, $|\xi|=1$. Consider a (branching ) geodesic trajectory  $\Psi_t (z)$ with $\pm t\in [0,  T]$ (one sign only) with $T\ge \epsilon_0$  and assume that for each $t$ indicated it meets $\partial X$ transversally i.e.
\begin{multline}
\dist (\uppi_x \Psi_t (x,\xi),\partial X) \le \epsilon
\implies\\
 |\frac{d\ }{dt}\dist (\uppi_x \Psi_t (x,\xi),\partial X)|\ge \epsilon \qquad \forall t: \pm t\in [0, mT].
\label{eqn-3-13}
\end{multline}
Also assume that
\begin{equation}
\dist (\uppi_x \Psi_t (x,\xi),\partial X)\ge \epsilon_0 \qquad \text{as\ \ }
t=0, \ \pm t=T.
\label{eqn-3-14}
\end{equation}
Let $\epsilon >0$ be a small enough constant, $Q$ be supported in $\epsilon$-vicinity of $(x,\xi)$ and $Q_1\equiv 1$ in $C_0\epsilon$-vicinity of $\Psi_t(x,\xi)$ as $t=\pm T$. Then operator  $(I-Q_1) e^{-it\Lambda }Q$ is negligible as $t=\pm mT$.
\end{theorem-foot}

\begin{proof}
Proof follows the proof of  Theorem~\ref{monsterbook-thm-8-5-11} of \cite{monsterbook} with $m=1$. 
\end{proof}

Adapting construction of the ``dependence set'' to our case, we arrive to the following

\begin{definition}\label{def-3-13}
\begin{enumerate}[label=(\roman*), wide, labelindent=0pt]
\item\label{def-3-13-i}
The curve $z(t)$ in $T^*Y$ is called a \emph{generalized geodesic billiard\/} if a.e. 
\begin{equation}
\frac{dz}{dt} \in K(z),
\label{eqn-3-15}
\end{equation}
where 
\begin{enumerate}[label=(\alph*), wide, labelindent=20pt]
\item \label{def-3-13-ia}
$K(z)= \{H_{g}(z)\}$, $g(z)$ is a metric form, if $z\in T^*(Y\setminus Z)$,
\item\label{def-3-13-ib}
$K(z) =\{H_{g}(z')\colon z'\in \iota^{-1}\iota z\}$,  if $z\in T^*Y|_Z$.
\end{enumerate}
\item\label{def-3-13-ii}
Let $\Psi_t(z)$ for $t\gtrless 0$ be a set of points  $z'\in T^*Y$ such that there exists generalized geodesic billiard $z(t')$ with $0 \lessgtr t' \lessgtr t$, such that $z(0)=z$ and $z(t)=z'$. Map $(z,t)\mapsto \Psi_t(z)$ is called a \emph{generalized (branching) billiard flow\/}.
\item\label{def-3-13-iii}
Point $z\in T^*Y$ is \emph{partially periodic\/} (with respect to $\Psi$) if for some $t\ne 0$ $z\in \Psi_t(z)$. Point $z\in T^*Y$ is \emph{completely periodic\/} (with respect to $\Psi$) if for some $t\ne 0$ $\{z\}= \Psi_t(z)$
\end{enumerate}
\end{definition}

Then we arrive immediately to
\begin{corollary}\label{cor-3-14}
Assume that $Z$ consists only of only outer edges. Also assume that the set of all partially periodic points 
is zero. 

Then $\N_h (\tau)$ and $\N_h (\tau) *\tau_+^{r-1}$ are approximated by the corresponding Tauberian expressions with $T\asymp h^{1-\delta}$  with errors  $o(h^{1-d})$ and $o(h^{1-d+r})$ respectively (as $\tau= 1$ and $h\to +0$).
\end{corollary}

\begin{proof}
Easy details are left to the reader.
\end{proof}

\chapter{Main results}
\label{sect-4}

\section{From Tauberian to Weyl asymptotics}
\label{sect-4-1}

Now we can apply the method of successive approximations as described in Section~\ref{monsterbook-sect-7-2} of \cite{monsterbook}, considering an unperturbed operator 
\begin{enumerate}[label=(\alph*), wide, labelindent=20pt]
\item 
As one in $\bR^{d}$, with the metrics, frozen at point $y$, if $\dist(y,Z)\ge h^{1-\delta}$.
\item
As one in the dihedral edge, with the metrics, frozen at point $(y',0)$, if $y=(y';s)$ with
$|s|=\dist(y,Z)\le h^{1-\delta}$,
\end{enumerate}
with the following modification: 

We calculate $\Lambda$ also this way,  applying successful approximations for both $\Delta $, when we solve 
$\Delta w=0$, $w|_Y=v$, and to $\partial_\nu $, when we calculate $\partial_\nu w|_Y$.

Then we prove that for operator $h\Lambda$ the Tauberian expression $\N_h^\T(1)$ for $\N_h^-(1)$ with $T=h^{1-\delta}$ coincides modulo $O(h^m)$ with the (generalized) Weyl expression
\begin{equation}
\N^\W_h \sim  \kappa_0 h^{-d} + \kappa_1 h^{1-d}+ \ldots,
\label{eqn-4-1}
\end{equation}
with the standard coefficient $\kappa_0$  and with $\kappa_1=\kappa_{1,Y\setminus Z}+\kappa_{1,Z}$, where $\kappa_{1,Y\setminus Z}$ is calculated in the standard way, for the smooth boundary, and
\begin{gather}
\kappa_{1,Z}=(2\pi)^{1-d}\omega_{d-1}\int _Z \varkappa (\alpha (y))\,dy\,,
\label{eqn-4-2}\\
\varkappa (\alpha) =
\int_1^\infty \int_{-\infty}^\infty \lambda ^{-d}
\Bigl( \mathbf{e}_\alpha (s,s,\lambda) - \pi ^{-1}(\lambda-1) \Bigr)\,dsd \lambda\,,
\label{eqn-4-3}
\end{gather}
$\mathsf{e}_\alpha (s,s',\lambda)$ is a Schwartz kernel of the spectral projector of $\hat{\Lambda}$ in the planar angle of solution $\alpha$ and  $\pi ^{-1}(\lambda-1)$  is a corresponding Weyl approximation.

\section{Main theorems}
\label{sect-4-2}

Thus we arrive to the corresponding asymptotics for $\N^-_h(1)$ and from them, obviously to asymptotics for 
$\N (\tau)$:

\begin{theorem}\label{thm-4-1}
Let $Y$  be a compact manifold with edges. Then the following asymptotics hold as $\tau\to+\infty$:
\begin{align}
&\N(\tau) = \kappa_0 \tau^{d} + O(\tau^{d-1}) 
\label{eqn-4-4}
\shortintertext{and for $r>0$}
&\N(\tau)*\tau_+^{r-1} = \Bigl(\sum_{m<r} \kappa_m \tau^{d-r}\Bigr)*\tau_+^{r-1} + O(\tau^{d-1}).
\label{eqn-4-5}
\end{align}
\end{theorem}

\begin{theorem}\label{thm-4-2}
Let $Y$  be a compact manifold with edges. Assume that $Z$ consists only of outer edges and that the set of all  points, which are partially periodic with respect to the generalized billiard  flow, has a measure~$0$.  Then the following asymptotics hold as $\tau\to+\infty$:
\begin{align}
&\N(\tau) = \kappa_0 \tau^{d} + \kappa_1 \tau^{d-1}+o(\tau^{d-1}) 
\label{eqn-4-6}
\shortintertext{and for $r>0$}
&\N(\tau)*\tau_+^{r-1} = \Bigl(\sum_{m \le r} \kappa_m \tau^{d-r}\Bigr)*\tau_+^{r-1} + o(\tau^{d-1}).
\label{eqn-4-7}
\end{align}
\end{theorem}
\section{Discussion}
\label{sect-4-3}

\begin{remark}\label{rem-4-3}
\begin{enumerate}[label=(\roman*), wide, labelindent=0pt]
\item\label{rem-4-3-i}
Even for standard ordinary non-branching billiards, billiard flow $\Psi_t$ could be multivalued. However, if through point $z\in T^*(Y\setminus Z)$ (where now $Y$ is a manifold, and $Z$ is its boundary) passes an infinitely long in the positive (or negative) time direction billiard trajectory, which always meets $Z$ transversally, and each finite time interval contains a finite number of reflections, then $\Psi_t(z)$ for $\pm t>0$ is single-valued. Points, which do not have such property, are called \emph{dead-end points\/}. For ordinary billiards the set of dead-end points has measure zero.

\item\label{rem-4-3-ii}
For branching billiards (with velocities $c_1,c_2$) we can introduce the notion of the dead-end point as well: it is if at least one of the branches either meets $Z$ non-transversally, or makes an infinite number of reflections on some finite time interval. As it was shown by Yu.~Safarov and D.~Vassiliev \cite{safarov:vassiliev}, if $c_1$ and $c_2$ are not disjoint (our case!), the set of dead-end billiards could have positive measure.
\end{enumerate}
\end{remark}

\begin{remark}\label{rem-4-4}
\begin{enumerate}[label=(\roman*), wide, labelindent=0pt]
\item\label{rem-4-4-i}
Checking non-periodicity assumption is difficult. But in some domains  it will be doable. F.e. assume that $Y=Y_1\cup Y_2$ globally is a domain of revolution, so $Z$ is a $(d-1)$-dimensional sphere.  Then the measure of  the set of dead-end billiards is $0$. 
\item\label{rem-4-4-ii}
Assume that neither $Y_1$, nor $Y_2$ contains closed geodesics.  Let $\varphi_j(\beta)$ be the length of the segment of geodesics in $Y_j$, with only ends on $Z$, where $\beta$ is the reflection angle:
\begin{figure}[h]
\centering
\begin{tikzpicture}[scale=1.2]

\fill[cyan!5] (0,0) circle (1.5);
\node at (-1,-.5) {$Y_j$};
\draw[cyan] (0, 1.5)--(0,0)--(1.2,-.9);
\draw[thick]  (0, 1.5)--(1.2,-.9);
\draw [thin, cyan] (-.7,1.5)--(.7,1.5);
\draw[thin,cyan] (.2,1.5) arc (0:-65:.2);
\node at (.3,1.3) {$\beta$};

\draw[thin,cyan] (0,.2) arc (90:-33:.2);
\node at (.4,.1) {$\varphi_j$};

\draw[very thick] (0,0) circle (1.5);
\node at (1.3,1.3) {$Z$};
\end{tikzpicture}
\label{fig-1}
\caption{Trajectories on the manifold of revolution }
\end{figure}
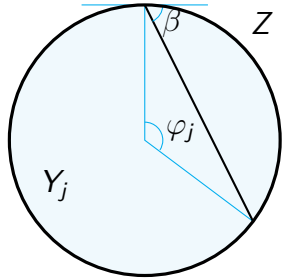
 Assume that $\varphi_j(\beta)$ are analytic and $\varphi_j(\beta)\to 0$ as $\beta\to +0$. Then the measure of  the set of partially periodic billiards is $0$.
 \end{enumerate}
\end{remark}

\begin{remark}\label{rem-4-5}
Our arguments hold  not only for compact $X$ but also for $X\subset \bR^{d+1}$ with the compact complement and with the metrics. stabilizing to Euclidean at infinity.
\end{remark}

\begin{remark}\label{rem-4-6}
\begin{enumerate}[label=(\roman*), wide, labelindent=0pt]
\item\label{rem-4-6-i}
In the next version of this paper we want to prove sharper asymptotics for domains with inner edges. To do this we need to understand, how singularities propagate near inner edges. One can prove that there are plenty of singularities, concentrated in  $Z\times \bR\ni (x,t)$ and $\{|\xi'|< \tau\}$. This is  similar to the Rayleigh waves. And, we hope, exactly like Rayleigh waves, those singularities do not prevent us from the sharper asympotics.

What we need to prove is that the singularities in $\{|\xi'|<\tau\}$, coming from $Y\setminus Z$ transversally to $Z$, reflect and refract but  leave $Z$ instantly. In other words, that these two kinds of waves are completely separate. It is what I am trying to prove now.

\item\label{rem-4-6-ii}
Let $\cK$ be the linear span of the corresponding eigenfunctions. We need to prove that  $\|\nabla w\|\ge \|w\|$  holds for $w=\hat{J}v$ with $v\in \cK^\perp$. One can prove easily that  $\|\nabla w\|= \|w\|$ for $w=\hat{J}v$ and eigenfunction $v$ (Proposition~\ref{prop-A-16}).
\end{enumerate}
\end{remark}

\begin{appendices}
\chapter{Planar toy-model}
\label{sect-A}

\section{Preparatory results}
\label{sect-A-1}

Here, in contrast to the whole article, $X= \{x\in\bR^2, x_1\ge |x_2|\cot (\alpha/2)\}$  is a planar angle of solution $\alpha \in (0,2\pi]$ with a boundary $Y=Y_1\cup Y_2$, $Y_{1,2}= \{x\colon x_1=|x_2|\cot(\alpha/2),\ \pm x_2<0\}$ and a bisector $Y_0 =\{x\colon x_2=0, x_1>0\}$, and $\Delta =-\partial_1^2-\partial_2^2$ is a positive Laplacian (so, for simplicity we do not write ``hat'').

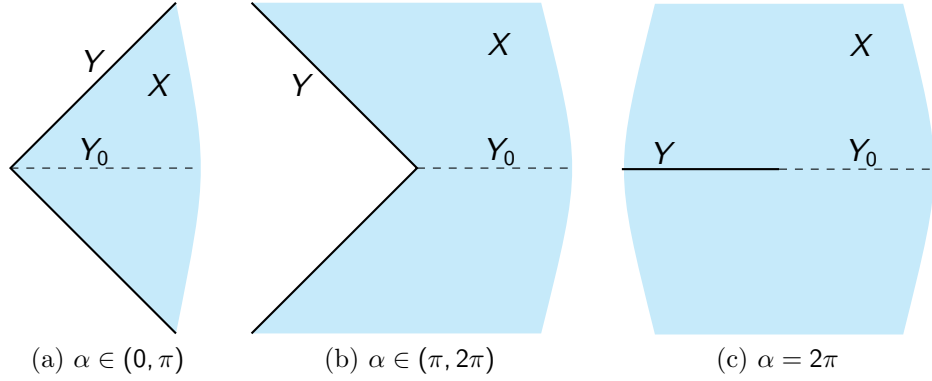
\begin{figure}[h]
\centering
\subfloat[$\alpha \in (0,\pi)$]{%
\begin{tikzpicture}[scale=1.1]
 \fill[cyan!20] (2,2)--(0,0)--(2,-2).. controls (2.4,0)..(2,2);
\draw[thick] (2,2)--(0,0)--(2,-2);
\draw[dashed] (0,0)--(2.2,0);
\node at (1,1.3) {$Y$};
\node at (1,.2) {$Y_0$};
\node at (1.8,1) {$X$};
\end{tikzpicture}}
\quad
\subfloat[$\alpha \in (\pi,2\pi)$]{
\begin{tikzpicture}[scale=1.1]
\fill[cyan!20] (1.5,2)--(-2,2)--(0,0)--(-2,-2)--(1.5,-2) .. controls (2,0)..(1.5,2);
\draw[thick] (-2,2)--(0,0)--(-2,-2);
\draw[dashed] (0,0)--(1.9,0);
\node at (1,.2) {$Y_0$};
\node at (-1.4,1) {$Y$};
\node at (1,1.5) {$X$};
\end{tikzpicture}}
\quad
\subfloat[$\alpha =2\pi$]{%
\begin{tikzpicture}[scale=1.1]
\fill[cyan!20] (1.5,2)--(-1.5,2)..controls (-2,0) ..(-1.5,-2)--(1.5,-2) .. controls (2,0)..(1.5,2);
\draw[thick] (-1.9,0)--(0,0);
\draw[dashed] (0,0)--(1.9,0);
\node at (1,.2) {$Y_0$};
\node at (-1.4,.2) {$Y$};
\node at (1,1.5) {$X$};
\end{tikzpicture}}

\caption{Proper and improper angles}
\label{fig-2}
\end{figure}

\begin{remark}\label{rem-A-1}
\begin{enumerate}[label=(\roman*), wide, labelindent=0pt]
\item\label{rem-A-1-i}
For $\alpha=\pi$ we have a regular half-plane $\{x\colon x_1>0\}$, and for $\alpha =2\pi$ we have a plane with the cut $\{x\colon x_1\le 0,\ x_2=0\}$.
\item
One can consider $\alpha >2\pi$ on the covering of $\bR^2$.
\end{enumerate}
\end{remark}

Consider real-valued\footnote{\label{foot-A-1} For complex-valued solutions then the main inequalities  with $w^2$ replaced by $|w|^2$ follow automatically.} solutions of
\begin{equation}
L w \coloneqq (\Delta +1)w=0
\label{A-1}
\end{equation}
and operators $J$, $\Lambda$: $w=Jv$ solves (\ref{A-1}) and $w|_Y=v$; $\Lambda v=-\partial_\nu w|_Y$ with $w=Jv$. Recall that $\nu$ is an inner normal to $Y$.

Observe that for any angle\footnote{\label{foot-A-2} Not necessary symmetric with respect to $x_1$-axis.}
\begin{equation}
2\iint _X  Lw\cdot  w_{x_1}\,dx_1dx_2= \int_{\partial X} \Bigl((w_{x_1}^2 -w_{x_2}^2-w^2) \nu_1 + 2 w_{x_1}w_{x_2} \nu_2 \Bigr)\,dr,
\label{A-2}
\end{equation}
where $dr$ is a Euclidean measure on $Y$, and the then similar formula holds with $x_1$ and $x_2$ permuted.
Then for solution of (\ref{A-1})
\begin{equation}
\int_{\partial X}
\Bigl((w_{x_1}^2 -w_{x_2}^2) (\nu_1 \ell_1 -\nu_2\ell_2) + 2 w_{x_1}w_{x_2}   (\nu_2 \ell_1 +\nu _1\ell_2)-w^2(\nu_1\ell_1+\nu_2\ell_2)\Bigr)\,dr=0.
\label{A-3}
\end{equation}
If on $\Gamma\subset \partial X$ $\ell_1=\nu_1$, $\ell_2=\nu_2$, then we can calculate invariantly as if $\ell_1=\nu_1=0$, $\ell_2=\nu_2=1$:
\begin{multline}
(w_{x_1}^2 -w_{x_2}^2) (\nu_1 \ell_1 -\nu_2\ell_2) + 2 w_{x_1}w_{x_2}  (\nu_2 \ell_1 +\nu _1\ell_2)-w^2(\nu_1\ell_1+\nu_2\ell_2)= \\
w_\nu^2-w_r^2 -w^2,
\label{A-4}
\end{multline}
where $w_r=\partial_r w$ and $w_\nu=\partial_\nu w$.

All these formulae hold not only for the original angle, but also for the smaller angle. Then let consider
as $X$ an upper half of the symmetric angle, $\partial X= Y_2\cup Y_0$, on $Y_0$ the integrand is
\begin{equation}
\cI\coloneqq  (w_{x_2}^2 -w_{x_1}^2 -w^2)\ell_2   + 2 w_{x_1}w_{x_2}   \ell_1
\label{A-5}
\end{equation}
with $\ell_1=\sin (\alpha/2)$, $\ell_2=-\cos(\alpha/2)$.

Consider different cases:

\emph{Antisymmetric case:\/} $w|_{Y_0}=0$, then $\cI= -w_{x_2}^2\cos(\alpha/2)$ and
\begin{equation}
\int _{Y_2} \bigl(w_\nu^2-w_r^2 -w^2\bigr)\,dr - \cos(\alpha/2) \int_{Y_0} w_{x_2}^2 \,dx =0.
\label{A-6}
\end{equation}

\emph{Symmetric case:\/} $w_{x_2}|_{Y_0}=0$, then $\cI= (w_{x_1}^2+w^2)\cos(\alpha/2)$ and
\begin{equation}
\int _{Y_2} \bigl(w_\nu^2-w_r^2-w^2\bigr)\,dr + \cos(\alpha/2)\int_{Y_0} (w_{x_1}^2+w^2)\,dx_1 =0.
\label{A-7}
\end{equation}

\begin{proposition}\label{prop-A-2}
Let $w$ satisfy \textup{(\ref{A-1})}. Let \underline{either} $\alpha \in (0,\pi]$  and $w$ is antisymmetric, \underline{or} $\alpha \in [\pi,2\pi)$  and $w$ is symmetric. Then\begin{equation}
\|\nabla w\|^2 \ge  \|w\|^2.
\label{A-8}
\end{equation}
\end{proposition}

\begin{proof}
In both cases $\int_{Y_2} (|\nabla w|^2-w^2)\,dr \ge \int_{Y_2} (w_\nu^2\, -w^2)\,ds \ge 0$. Applying this inequality to the angle, shifted by $t$ along $x_1$, and integrating by $t\in (0,\infty)$,  we obtain a double integral (divided  by $\sin (\alpha/2)$).

Moreover, one can see easily, that this inequality is strict unless $w=0$.

\end{proof}

Similarly, if instead of multiplying by $(\nu_1 w_{x_1}+\nu_2 w_{x_2})$ we multiply by $(x_2 w_{x_1} - x_1w_{x_2})$, then extra terms in the double integral will be $\pm w_{x_1} w_{x_2}$ and they cancel one another. However, on $Y$ we get $x_2=\nu_1 r$, $x_1=-\nu_2 r$ and therefore contribution of $Y_2$ will be as in above with extra factor $r$:
\begin{equation}
\int _{Y_1} \bigl(w_\nu^2-w_r^2-w^2\bigr)\,rdr .
\label{A-9}
\end{equation}
On $Y_0$ we get extra factor $x_1=r$, but not $\nu_2=-\cos(\alpha/2)$, and we arrive to

\emph{Antisymmetric case:\/} $w|_{Y_0}=0$, then $\cI= w_{x_2}^2x_1$ and
\begin{equation}
\int _{Y_2} \bigl(w_\nu^2-w_r^2-w^2\bigr)\,rdr + \int_{Y_0} w_{x_2}^2 \,x_1dx_1 =0.
\label{A-10}
\end{equation}

\emph{Symmetric case}: $w_{x_2}|_{Y_0}=0$, then $\cI= (-w_{x_1}^2-w^2)x_1$ and
\begin{equation}
\int _{Y_2} \bigl(w_\nu^2-w_r^2-w^2\bigr)\,rdr - \int_{Y_0} (w_{x_1}^2+w^2)\,x_1dx_1 =0.
\label{A-11}
\end{equation}

Let us explore dependence $\Lambda=\Lambda(\alpha)$ on $\alpha$. Observe first that
\begin{gather}
\iint \bigl(\nabla w\cdot \nabla w'+ww'\bigr)\,dxd_1dx_2 =
\int Lw\cdot w' - \int_{\partial X} \partial_\nu w\cdot w'\,dr
\label{A-12}\\
\intertext{where $(r,\theta)$ are polar coordinates and therefore $dr$ is an Euclidean measure on $Y$. It implies}
(\Lambda v,v')_Y = \iint \bigl( \nabla w\cdot \nabla w' + ww'\bigr)\,dx_1dx_2,
\label{A-13}
\end{gather}

for $w=Jv$, $w'=Jv'$. Therefore
\begin{claim}\label{A-14}
$\Lambda $ is symmetric and nonnegative operator in $\sL^2(Y)$.
\end{claim}

Consider $X=X(\alpha)$, $Y=Y(\alpha)$, $\Lambda=\Lambda(\alpha)$ and keep $w$ independent on $\alpha$. Let us replace $\alpha$ by $\alpha+\updelta\alpha$ etc. Then for a symmetric $X$ we have
$\updelta v = -r(\partial_\nu w) \updelta\alpha= \frac{1}{2}r(\Lambda v) \updelta\alpha$ and it follows from
(\ref{A-13}) that
\begin{gather}
((\updelta \Lambda)v,v)_Y + 2(\Lambda v, \updelta  v)_Y  = \frac{1}{2}\int_Y \bigl( |\nabla w|^2 + |w|^2\Bigr)\,rdr\times \updelta\alpha
\notag
\shortintertext{and therefore}
((\updelta \Lambda)v,v)_Y =  -\frac{1}{2}\int_Y \bigl( w_\nu^2 -w_r^2 - |w|^2\Bigr)\,rdr \times \updelta\alpha.
\label{A-15}
\end{gather}
Combining with (\ref{A-10}) and (\ref{A-13}) we arrive to

\begin{proposition}\label{prop-A-3}
\begin{enumerate}[label=(\roman*), wide, labelindent=0pt]
\item\label{prop-A-3-i}
On symmetric functions $\Lambda(\alpha)$ is monotone increasing function of $\alpha$.
\item\label{prop-A-3-ii}
On antisymmetric functions $\Lambda(\alpha)$ is monotone inreasing function of $\alpha$.
\end{enumerate}
\end{proposition}

Let us identify $Y$ with $\bR\ni s$, $s=\mp r$ on $Y_{1,2}$ respectively.

\begin{proposition}\label{prop-A-4}
\begin{enumerate}[label=(\roman*), wide, labelindent=0pt]
\item\label{prop-A-4-i}
On symmetric functions $\Lambda(\pi)=(D_s^2+I)^{\frac{1}{2}}$.
\item\label{prop-A-4-ii}
On antisymmetric functions $\Lambda(2\pi)\ge (D_s^2+I)^{\frac{1}{2}}$.
\end{enumerate}
\end{proposition}

\begin{proof}
Statement \ref{prop-A-4-i} is obvious. Statement \ref{prop-A-4-ii} follows from the fact that on antisymmetric function $v$  $\Lambda(2\pi)v $ coincides with $\Lambda (\pi)v^0$, restricted to $\{x_1<0\}$, where $v^0$ is $v$, extended by $0$ to $\{x_1>0\}$.
\end{proof}

Therefore, combining Propositions~\ref{prop-A-3} and ~\ref{prop-A-4} we conclude that

\begin{corollary}\label{cor-A-5}
\begin{enumerate}[label=(\roman*), wide, labelindent=0pt]
\item\label{cor-A-5-i}
On symmetric functions $\Lambda(\alpha )\ge (D_s^2+I)^{\frac{1}{2}}$ for $\alpha \in [\pi,2\pi]$.
\item\label{prop-A-5-ii}
On antisymmetric functions $\Lambda(\alpha)\ge (D_s^2+I)^{\frac{1}{2}}$ for $\alpha \in (0,2\pi]$.
\end{enumerate}
\end{corollary}

\begin{remark}\label{rem-A-6}
One can prove easily, that inequalities are strict for $\alpha \in (\pi,2\pi]$, $\alpha \in (0,2\pi]$ respectively.
\end{remark}

Now we want to finish general arguments and to prove inequality (\ref{A-8}) for antisymmetric $w$ and $\alpha \in (\pi,2\pi]$. It will be more convenient to use polar coordinates $(r,\theta)$ and notations
$\cY_\beta=\{(r,\theta)\colon:\theta=\beta\}$, $\cX_{\beta_1,\beta_2}=\{(r,\theta)\colon:\beta_1\le \theta\le \beta_2\}$.
Here and below $\beta_*\in [-\alpha/2,\alpha/2]$. Recall that
\begin{equation}
L=-\partial_r^2-r^{-1}\partial_r -r^{-2}\partial_\theta^2 +1.
\label{A-16}
\end{equation}

\begin{proposition}\label{prop-A-7}
\begin{enumerate}[label=(\roman*), wide, labelindent=0pt]
\item\label{prop-A-7-i}
Let $w$ satisfy  equation \textup{(\ref{A-1})} in $X$.  Then
\begin{equation}
\cI(\beta )\coloneqq \int _{\cY_{\beta}} \Bigl[ r^{-2} w_\theta  ^2-w_r^2-w^2\Bigr]\,rdr
\label{A-17}
\end{equation}
does not depend on $\beta$.

\item\label{prop-A-7-ii}
Therefore
\begin{equation}
\cJ(\beta_1,\beta_2)\coloneqq \iint _{\cX_{\beta_1,\beta_2}} \Bigl[r^{-2} w_\theta  ^2-w_r^2-w^2\Bigr]\,rdr d\theta
\label{A-18}
\end{equation}
depends only on $\beta_2-\beta_1$ and therefore is proportional to it.
\end{enumerate}
\end{proposition}

\begin{proof}
One proves \ref{prop-A-7-i}  by analyzing $-\iint _{\cX_{\beta_1,\beta_2}} Lw\cdot\partial_\theta w\, dxdy$ (which actually was done before, since $w_\theta = -x_2w_{x_1}+ x_1w_{x_2}$.

To prove \ref{prop-A-7-ii} observe that
$\partial_\beta \cJ(\beta_1,\beta)=\cI(\beta)$.
\end{proof}

\begin{proposition}\label{prop-A-8}
\begin{enumerate}[label=(\roman*), wide, labelindent=0pt]
\item\label{prop-A-8-i}
Function
\begin{equation}
\cJ(\beta_1,\beta_2)\coloneqq \int _{\cX_{\beta_1,\beta_2}} w^2 r^{-1}\,drd\theta
\label{A-19}
\end{equation}
with  fixed $\beta_{1,2}=\beta\mp \sigma$ is convex with respect to $\beta$ (if $\sigma>0$).
\item\label{prop-A-8-ii}
Further, if $w$ is either symmetric or antisymmetric, then
it   reaches minimum as $\beta=0$ (i.e. $\cX_{\beta_1,\beta_2}$ is symmetric with respect to $Y_0$).
\end{enumerate}
\end{proposition}

\begin{proof}
\begin{enumerate}[label=(\roman*), wide, labelindent=0pt]
\item\label{pf-A-8-i}
Consider
\begin{multline}
0= \iint_{\cX_{\beta_1,\beta_2}} Lw\cdot w \, r drd\theta = \\
\iint_{\cX_{\beta_1,\beta_2}} (w_r^2 +r^{-2}w_\theta^2 + w^2)\, r drd\theta
+\cI'(\beta_1)-\cI'(\beta_2)
\label{A-20}
\end{multline}
with
\begin{equation}
\cI'(\beta)=\int_{\cY_{\beta}} ww_\theta \,r^{-1} dr=\partial_\beta \cI(\beta),\qquad
\cI(\beta)\coloneqq \int_{\cY_{\beta}} w^2 \,r^{-1} dr.
\label{A-21}
\end{equation}
Observe that the first term is positive.   Then $\cI'(\beta_2) -\cI'(\beta_1)>0 $; on the other hand, it is the second derivative of $\cJ(\beta_1,\beta_2)$ with respect to $\beta$.

\item\label{pf-A-8-ii}
Moreover, for both symmetric and antisymmetric $w$  $\cI(\beta_2)-\cI(\beta_1)=0$. And the difference $\cI(\beta_2) -\cI(\beta_1)=0$ for $\beta=0$.
\end{enumerate}
\end{proof}

\begin{corollary}\label{cor-A-9}
Since $w_\theta $ satisfies the same equation and is antisymmetric (symmetric) respectively, the same conclusions \ref{prop-A-8-i}, \ref{prop-A-8-ii} hold for
$\cJ\coloneqq \int _{\cX_{\beta_1,\beta_2}} w_\theta^2 r^{-1}\,rdrd\theta$.

Then in virtue of Proposition~\ref{prop-A-7}\ref{prop-A-7-ii}  the same conclusions
\ref{prop-A-8-i}, \ref{prop-A-8-ii} hold for
$\cJ\coloneqq \iint _{\cX_{\beta_1,\beta_2}} (w_r^2 + w^2)\,rdrd\theta$.
\end{corollary}

Next, observe that $ L r\partial_r w = 2\Delta w = -2 w$ and if we use the same arguments, as in the proof of Proposition~\ref{prop-A-8}\ref{prop-A-8-ii} for $r\partial_r w $, then instead of the first term in  (\ref{A-20}) we get
\begin{gather}
\iint_{\cX_{\beta_1,\beta_2}} \bigl( (rw_r)_r^2 +w_{r\theta}^2 + (rw_r)^2 -w^2\bigr)\, r drd\theta,
\label{A-22}\\
\intertext{where an additional last term appears as}
\iint_{\cX_{\beta_1\beta_2}}  2w \partial_r w\cdot r^2drd\theta= -\iint_{\cX_{\beta_1\beta_2}} w^2\,rdrd\theta.
\notag
\end{gather}
Consider last two terms and skip integration by $d\theta$; plugging $w=r^{-3/2}u$ with $u(0)=0$, we arrive to
\begin{multline*}
\int \bigl(u_r - \frac{3}{2}r^{-1}u)^2 - r^{-2}u^2\bigr)\,dr=
\int \bigl(u_r^2 -3r^{-1}u_ru +\frac{5}{4} r^{-2}u^2\bigr)\,dr =
\int \bigl(u_r^2 - \frac{1}{4}r^{-2}u^2\bigr)\,dr
\end{multline*}
which is again nonnegative term. Then we arrive to

\begin{corollary}\label{cor-A-10}
The same conclusions \ref{prop-A-8-i} and \ref{prop-A-8-ii} of Proposition~\ref{prop-A-8} hold for $\cJ(\beta_1,\beta_2)$ with $w$ replaced by $rw_r$, i.e.
$\cJ(\beta_1,\beta_2)\coloneqq \iint_{\cX_{\beta_1,\beta_2}}w_r^2\,rdrd\theta$.

Then in virtue of Proposition~\ref{prop-A-7}\ref{prop-A-7-ii}  the same conclusions
\ref{prop-A-8-i}, \ref{prop-A-8-ii} hold for
\begin{multline}
2\iint_{\cX_{\beta_1,\beta_2}} |\partial_rw|^2\,rdrd\theta  + \\
\iint_{\cX_{\beta_1,\beta_2}} (r^{-2}|\partial_\theta w|^2 - |\partial_rw|^2 -|w|^2|)\,rdrd\theta=
\iint_{\cX_{\beta_1,\beta_2}} (|\nabla w|^2-|w|^2)\,dxdy.
\end{multline}
\end{corollary}

Now we can prove
\begin{proposition}\label{cor-A-11}
Let $\alpha \in (\pi,2\pi]$. Then for both symmetric and antisymmetric $w$ \textup{(\ref{A-8})} holds.
\end{proposition}

\begin{proof}
Indeed, assume that it is not the case: $\iint_X (|\nabla w|^2-w^2)\,dxdy <0$ for some $w$. Then due to Corollary~\ref{cor-A-10} the same is true for $X$ replaced by $\cX_{\beta_1,\beta_2}$ with
$\beta_{1,2}=\mp (2\pi -\alpha)/2$ and the same $w$. Then it is true for the sum of these to expressions (with $X==X_{-\alpha/2,\alpha/2}$ and $\cX_{\beta_1,\beta_2}$), which is the sum of the same expressions for the half-planes $X_{\alpha/2-\pi,\alpha/2}$ and $X_{-\alpha/2,\pi-\alpha/2}$. However, for half-planes (\ref{A-8}) holds.
\end{proof}

Let $P_\tau=\uptheta(\tau-\Lambda)$.

\begin{proposition}\label{prop-A-12}
\begin{enumerate}[label=(\roman*), wide, labelindent=0pt]
\item\label{prop-A-12-i}
Let $\alpha \in [\pi,2\pi]$. Then  for any $\tau>1$ for  $w=Jv$, $v\in \Ran (I-P_\tau)$,
\begin{equation}
\|\nabla w\|^2\ge (1+\delta)\|w\|^2
\label{A-24}
\end{equation}
with $\delta=\delta(\tau)>0$.

\item\label{prop-A-12-ii}
Let $\alpha \in (0,\pi]$. Then for any  $\tau>1$ for antisymmetric $w=Jv$, $v\in \Ran (I-P_\tau)$,
\textup{(\ref{A-24})} holds.
\end{enumerate}
\end{proposition}

\begin{proof}
Observe first that 
\begin{claim}\label{A-25}
$\{v\in \Ran(I-P_\tau) \colon \|v\|_Y=1,\,  \colon \|\nabla Jv\|^2\le (1+\delta')\|Jv\|^2 \}$ is a compact set in $\sL^2(Y)$ for $\delta'=\delta'(\tau)>0$.
\end{claim}
Indeed, in the zone $\{x\colon |x|\ge R\}$ we can apply semiclassical arguments with $h\coloneqq R^{-1}$ after scaling $x\mapsto R^{-1}x$.

Since in both cases (\ref{A-8}) holds with a strict inequality for $w\ne 0$), we arrive to both Statements~\ref{prop-A-12-i} and \ref{prop-A-12-ii}. 
\end{proof}

\section{Spectrum}
\label{sect-A-2}

The above results are sufficient for our needs, for $\alpha \in (\pi,2\pi]$. However we would like to explore the case of $\alpha \in (0,\pi)$ and even $\alpha \in (\pi,2\pi]$ in more depth.

\begin{corollary}\label{cor-A-13}
\begin{enumerate}[label=(\roman*), wide, labelindent=0pt]
\item\label{cor-A-13-i}
Let $\alpha \in [\pi,2\pi]$. Then  $\Spec(\Lambda)=[1,\infty)$ and it is continuous.

\item\label{cor-A-13-ii}
Let $\alpha \in (0,\pi]$. Then  $\Spec(\Lambda_\asym)=[1,\infty)$ and it is continuous, where $\Lambda_\sym$ and $\Lambda_\asym$ denote the restriction of $\Lambda$ to the spaces of symmetric and antisymmetric functions, correspondingly.
\end{enumerate}
\end{corollary}

\begin{proof}
We already know that the that essential spectrum of $\Lambda $ is $[1,\infty)$.
We also know that in the case~\ref{cor-A-13-i} $\Lambda > I$ and in the case~\ref{cor-A-13-ii} $\Lambda_\asym >  I$.
Therefore $1$ is not an eigenvalue. Continuity of the spectrum follows from 
$(i[\Lambda,Q]v,v) \ge \delta \|v\|^2$ for $v\in \Ran (I-P_\tau)$, $v$ is antisymmetric in the case \ref{cor-A-13-ii}, which is due to Statements~\ref{prop-A-12-i} and \ref{prop-A-12-ii} of Propostion~\ref{prop-A-12}. 
\end{proof}

\begin{remark}\label{rem-A-14}
Paper~\cite{KP}  is dealing mainly with the eigenvalues of $\Delta_2$ in the planar sector under Robin boundary condition $(\partial_\nu +\gamma)w|_Y=0$, $\gamma>0$\,\footnote{In that paper $\alpha $ is a half-angle, and $\nu$ is a unit external normal. Below we refer to this paper using our notations.}.  Then eigenvalues $\tau$ of $\Lambda$ and eigenvalues $\mu $ of that problem are related through Birman-Schwinger principle and scaling:
$\mu_k = -\tau_k^{-2}\gamma^2$. Some of the results:
\begin{enumerate}[label=(\roman*), wide, labelindent=0pt]
\item\label{rem-A-14-i}
Theorem~3.1  states that $(-\infty, -\gamma^2)$ contains only discrete spectrum of such operator and it is finite. 

\item\label{rem-A-14-ii}
Theorem~2.3 states that for $\alpha \in (0,\pi)$   the bottom eigenvalue $-\gamma^2/\sin^2(\alpha/2)$ is simple and the corresponding eigenfunction is 
$\exp (-\gamma x_1/\sin(\alpha/2)$. 

\item\label{rem-A-14-iii}
Theorem~3.6 states that for $\alpha \in [\frac{\pi}{3},\pi)$ there is no other eigenvalues in $(0,1)$, while Theorem 4.1 implies that the number of such eigenvalues is $\asymp \alpha^{-2}$ as $\alpha\to 0$\,\footnote{\label{foot-yz} In fact, the compete asymptotic expansion of the eigenvalues is derived in Theorem~4.16 of \cite{KP}.}.
\end{enumerate}
\end{remark}

Then we conclude  that

\begin{corollary}\label{cor-A-15}
\begin{enumerate}[label=(\roman*), wide, labelindent=0pt]
\item\label{cor-A-15-i}
Interval $(0,1)$ contains only discrete spectrum of $\Lambda_\sym$ which is finite.

\item\label{cor-A-15-ii}
For $\alpha \in (0, \pi)$ the bottom eigenvalues is $\sin(\alpha/2)$ and the corresponding eigenfunction is 
$\exp (-x_1)$.
\end{enumerate}
\end{corollary}

The discrete spectrum  would not prevent us from the extending our main results to $\alpha\in (0,\pi)$. Even (possible) eigenvalue $1$ on the edge of the essential spectrum would not be an obstacle.
However eigenvalues embedded into $(1,\infty)$ are an  obstacle (see Proposition~\ref{prop-A-16}).

\begin{proposition}\label{prop-A-16}
If $w_p=Jv_p$ where $v_p$ are eigenfunctions of $\Lambda$, corresponding to eigenvalues $\tau_p$, and $\tau_j=\tau_k$, then
\begin{equation}
(\nabla w_j,\nabla w_k) -(w_j,w_k)=0. 
\label{A-26}
\end{equation}
In particular, 
\begin{equation}
\|\nabla w_j,\|^2 -\|w_j\|^2=0. 
\label{A-27}
\end{equation}
\end{proposition}

\begin{proof}
It follows from equality (\ref{eqn-3-2}) for $Q=x_1D_1+x_2D_2 +i/2$ and $([Q,\Lambda]v_j,v_k)= (\Lambda v_j,Qv_jk)-(Qv_j,\Lambda v_k)=0$ for eigenfunction $v_j$ $v_k$ provided $\tau_j=\tau_k$.
\end{proof}

To extend the main sharp spectral asymptotics to operators in domains with inner edges one needs to prove the first following

\begin{conjecture}\label{conj-A-17}
For any $\tau>1$ and for any $w=Jv$ with symmetric $v\in \Ran (I-P_\tau)$ estimate \textup{(\ref{A-24})} holds.
\end{conjecture}

\begin{remark}\label{rem-A-18}
\begin{enumerate}[label=(\roman*), wide, labelindent=0pt]
\item\label{rem-A-18-i}
Recall that this is true for $\alpha \in [\pi,2\pi]$ and, also, for $\alpha\in (0,\pi)$ and antisymmetric $v$. So, only the case of $\alpha\in (0,\pi)$ and symmetric $v$ needs to be covered.

\item\label{rem-A-18-ii}
So far it is unknown, if in the the case of $\alpha\in (0,\pi)$ $\Lambda_\sym$ has eigenvalues embedded into continuous spectrum $(1,\infty)$ or on its edge.

\item\label{rem-A-18-iii}
Also it is unknown in any case, if the continuous spectrum is absolutely continuous (i.e. that the singular continuous spectrum is empty).
\end{enumerate}
\end{remark}

\end{appendices}

\bibliographystyle{amsplain}

\end{document}